\documentclass[12pt,oneside,reqno]{amsart}

\usepackage{amsmath}
\usepackage[sort]{natbib}
\usepackage{amsthm,amssymb,mathrsfs,mathtools}
\usepackage[dvipsnames,x11names,svgnames]{xcolor}
\usepackage{dsfont}

\usepackage[colorlinks=true,linkcolor=DarkGreen,citecolor=DarkGreen, urlcolor=DarkGreen,pagebackref]{hyperref}

\usepackage{geometry}

\newcommand{\Z}{\mathbb{Z}}

\newcommand{\N}{\mathbb{N}}

\newcommand{\E}{\mathbb{E}}
\newcommand{\F}{\mathbb{F}}

\newcommand{\T}{\mathcal{T}}

\DeclareMathOperator{\VCdim}{VCdim}
\DeclareMathOperator{\Cay}{Cay}
\DeclareMathOperator{\RC}{RC}

\newcommand{\geqs}{\geqslant}
\newcommand{\leqs}{\leqslant}

\newtheorem{theorem}{Theorem}
\newtheorem{lemma}[theorem]{Lemma}

\newtheorem{conjecture}[theorem]{Conjecture}

\theoremstyle{definition}
\newtheorem{definition}{Definition}[section]

\theoremstyle{remark}
\newtheorem{remark}[definition]{Remark}

\begin{document}
 
\title{The VC-dimension of random subsets of finite groups}

\author{Brad Rodgers}
\email{\url{brad.w.rodgers@gmail.com}}
\address{Department of Mathematics and Statistics, Queen's University, University Avenue, K7L 3N6, Kingston, Canada}
\thanks{BR is partially supported by an NSERC grant.}

\author{Anurag Sahay}
\email{\url{anurag.sahay@duke.edu}}
\address{Department of Mathematics, Duke University, 120 Science Drive, Durham NC 27708, USA}
\thanks{AS was partially supported by Trevor Wooley's start-up funding at Purdue University and by the AMS-Simons Travel Grant at the time this article was written.}

\begin{abstract}
For a random subset of a finite group $G$ of cardinality $N$, we consider the VC-dimension of the family of its translates (equivalently the VC-dimension of a random Cayley graph) and prove a law of large numbers as $N\rightarrow\infty$. This answers a question of McDonald--Sahay--Wyman.
\end{abstract}

\maketitle

\section{Introduction}
\label{sec: intro}

\subsection{Random subsets}
The purpose of this note is to give an estimate which holds asymptotically almost surely for the VC-dimension associated to random subsets of finite groups, or equivalently for random Cayley graphs. In order to fix our terminology we recall some key definitions.

A \emph{set system} $\mathcal{F}$ is a family of subsets of a ground set $\Omega$. \emph{The VC-dimension} of the set system $\mathcal{F}$, named after Vapnik and Chervonenkis \cite{VC}, roughly quantifies how ``rich" the set system $\mathcal{F}$ is. It has its origins in statistics and computational learning theory \cite{shalev2014understanding}, but has also become an important concept in combinatorics and model theory \cite{conantpillayterry,sisask,conantpillay,zarankiewicz,foxpachsuk,foxpachsuk2,iosevichphamsengertait,terrywolf} (see also \cite{terrywolfnotVC}). 
 It is defined using the following set up. For a subset $U$ of $\Omega$, the \emph{restriction} of $\mathcal{F}$ on $U$ is the family
\[
\mathcal{F}|_U:=\{F \cap U:\, F \in \mathcal{F}\}.
\]
We say that $U$ is \emph{shattered} by $\mathcal{F}$ if $\mathcal{F}|_U$ consists of all $2^{|U|}$ subsets of $U$. Furthermore instead of writing  $K \in \mathcal{F}|_U$ we will sometimes say that the subset $K$ is \emph{cut out} from $U$ by $\mathcal{F}$. The \emph{VC-dimension} of the $\mathcal{F}$, denoted $\VCdim(\mathcal{F})$ is then defined to be the maximum cardinality of a shattered subset of $\Omega$.

If $G$ is a finite group written multiplicatively and $A$ is some subset of $G$, we define the set system $\T_A$ to be the family $\{tA:\, t \in G\}$ of left translates of $A$, with ground set $G$. Furthermore when there is no chance of confusion we will use the abbreviation
\[
\VCdim(A) := \VCdim(\T_A).
\]
Recently several authors (see e.g. \cite{conantpillayterry,sisask,conantpillay,alonfoxzhao, MSW,terrywolf}) have used or studied the properties of $\VCdim(A)$ (or closely related notions, see Remark \ref{remark: alternate_def}) in a variety of ambient groups.

In this paper we will be interested in estimating $\VCdim(A)$ when the subset $A$ has been chosen randomly from a large finite group $G$; in particular, we will consider Bernoulli sampled subsets. For a finite group (or set) $G$ of cardinality $N$, we say that a random subset $A \subseteq G$ is \emph{Bernoulli sampled with parameter $p \in [0,1]$} if for all $x$ in the group, the events $x \in A$ are independent and each occurs with probability $p$. In particular the distribution of $A$ is given by
\[
\Pr[A=W] = p^{|W|} (1-p)^{N-|W|},
\]
for any deterministic $W \subseteq G$.

Our main result is as follows.

\begin{theorem}
\label{thm: vcdim_translates}
Let $0 < p < 1$ be a fixed parameter and set $r = [\min(p,1-p)]^{-1}$. Let $G$ be a group of cardinality $N$ and let $A$ be a random subset of $G$ that is Bernoulli sampled with parameter $p$. Then for any $\eta >0$ we have
\[
|\VCdim(A) - \log_r N| \leqs 10 \log_r \log_r N,
\]
with probability $1-O(1/N^\eta)$. Here and throughout $\log_r$ refers to logarithms to the base $r> 1$.
\end{theorem}

Thus in particular, asymptotically almost surely
\[
\VCdim(A) = (1+o(1)) \log_r N.
\]

In this theorem, the implicit constant may depend on $p$ and $\eta$ but is uniform over groups $G$. The constant $10$ is just chosen to be sufficiently large and we have made no attempt to optimize it. 

\begin{remark} 
\sloppy Note that a standard subset counting argument yields that $0 \leqs \VCdim(A) \leqs \log_2 N$ for any subset $A \subseteq G$. Thus, in the balanced case $p = 1/2$, we have that the VC-dimension is likely to be essentially as large as possible.
\end{remark}

\begin{remark} 
Theorem \ref{thm: vcdim_translates} should hold even if $p = p(N)$ tends to zero sufficiently slowly with $N$. An estimate of this sort can likely be extracted from the proof below, however it is not clear how quickly $p(N)$ can be allowed to decrease with such an estimate remaining true, and we do not pursue this question further. 
\end{remark}

\begin{remark}
\label{remark: alternate_def}
Our definition of $\VCdim(A)$ differs slightly from that adopted by Sisask \cite{sisask}, who uses the family $\{tA \cap A:\, t \in A\cdot A^{-1}\}$ in place of $\{tA:\, t\in G\}$. Sisask is motivated by additive combinatorics in his definition but observes \cite[Proposition~4.1]{sisask} that the VC-dimensions of these two families differ by at one most $1$ however, so an estimate like Theorem \ref{thm: vcdim_translates} remains true under either convention.
\end{remark}

\begin{remark}
There also exist other conventions for the definition of VC-dimension for subsets of groups like the symmetric group -- see \cite{raz}. We will not use that definition of VC-dimension here however.
\end{remark}

\subsection{Random Cayley graphs}

This result can also be interpreted in terms of Cayley graphs.

We work with directed graphs on a finite set of vertices. For a graph $\Gamma$, we identify the graph with its vertex set, writing $u \in \Gamma$ to mean that $u$ is a vertex in $\Gamma$. Further, we use the notation $u \to v$ in $\Gamma$ to mean that there is an edge from the vertex $u$ to the vertex $v$. We allow loops but disallow multiple edges (having both $u \to v$ and $v \to u$ is not considered a multiple edge). 

The VC-dimension of a graph $\Gamma$ is defined to be the VC-dimension of the set system on $\Gamma$ given by the family of neighborhoods, 
\[ \mathcal{N}_\Gamma = \{ N(v) : v \in \Gamma \}, \]
where $N(v) = \{ w \in \Gamma : v \to w \text{ in } \Gamma \}$. We denote this by $\VCdim(\Gamma):=\VCdim(\mathcal{N}_\Gamma)$. See \cite{anthonybrightwellcooper,MSW,foxpachsuk,foxpachsuk2,PSTT} for this and closely related definitions. 


If $G$ is a finite group and $A \subseteq G$ is a subset, recall that the \emph{Cayley (di)graph} $\Gamma = \Cay(G,A)$ is the graph with vertex set $G$ such that $u \to v$ in $\Gamma$ if there exists an $a \in A$ such that $v = ua$. It thus follows that $\VCdim(\Gamma) = \VCdim(A)$ since $\mathcal{N}_\Gamma = \mathcal{T}_A$.

For a fixed group $G$, let us form a random graph $\Gamma = \Cay(G,A)$ by letting $A$ be a random subset of $G$ that is Bernoulli sampled with parameter $p$. We shall denote this by $\Gamma \sim \RC(G,p)$. Theorem \ref{thm: vcdim_translates} can therefore be stated in the following way.

\begin{theorem}
\label{thm: vcdim_graphs}
Let $0 < p < 1$ be a fixed parameter and set $r = [\min(p,1-p)]^{-1}$. Let $G$ be a group of cardinality $N$ and let $\Gamma \sim \RC(G,p)$ be a random Cayley graph on $G$. Then for any $\eta>0$
\[
|\VCdim(\Gamma) - \log_r N| \leqs 10 \log_r \log_r N
\]
with probability $1-O(1/N^\eta)$.
\end{theorem}

Thus, as before, $\VCdim(\Gamma) =(1+o(1)) \log_r N$ asymptotically almost surely. In this form, our result answers a question raised by McDonald--Sahay--Wyman in \cite[Section~5.3]{MSW}.

\begin{remark}
In \cite{anthonybrightwellcooper}, the VC-dimension of a graph is defined in terms of the family of \emph{closed} neighborhoods $\{N(v) \cup \{v\}:\, v \in \Gamma\}$ in contrast to our definition using open neighborhoods. We expect Theorem~\ref{thm: vcdim_graphs} should hold in this case also; this would correspond, in Theorem~\ref{thm: vcdim_translates}, to sampling the set $A$ randomly except for the identity element which is always included in $A$. 
\end{remark}

\begin{remark}
This result can be compared to the VC-dimension of random Erd\H{o}s-R\'enyi  graphs \cite{anthonybrightwellcooper} where a similar result is found, at least for small $p$. (By contrast, a rather different behavior for the VC-dimension emerges when considering a randomly sampled set family \cite{ycartratsaby}.)
\end{remark}

\begin{remark}
In the case that $G$ is a finite abelian group (now written additively), several authors \cite{green,greenmorris,konyaginshkredov,mrazovic}, have considered the \emph{Cayley sum graph}\footnote{Strictly speaking, Mrazovi\'c \cite{mrazovic} does not assume the group is abelian, so his graphs should be called Cayley product graphs.} $\Gamma^+$ associated to $G$ and a random subset $A$. This is an undirected graph with vertex set consisting of the elements of $G$, with an edge between $x$ and $y$ if $x+y \in A$. It is easy to see that\footnote{It follows from the observation that the neighborhood of $v$ in $\Gamma^+$ is precisely the neighborhood of $-v$ in $\Gamma$.} $\VCdim(\Gamma^+) = \VCdim(A)$ so that Theorem \ref{thm: vcdim_translates} applies to these graphs as well.
\end{remark}

\begin{remark} \label{rem: graphmodels}
There are graph models in the literature other than the ones discussed above that are also called random Cayley graphs; see e.g. \cite{alon,alonroichman,conlonfoxphamyepremyan}. Two variations are of particular interest. In one, the random Cayley graph is forced to have fixed regularity by randomly sampling all subsets $S \subseteq G$ of cardinality $d = pN$. In the other, the generating set of the Cayley graph is preconditioned to be symmetric, for example by using $S \cup S^{-1}$ as the generating set instead of $S$. It would be interesting to study the analogues of Theorem~\ref{thm: vcdim_graphs} for these models; the authors expect a similar result to hold.
\end{remark}

\subsection{Motivation} \label{subsec: numtheory}

Our motivation for Theorems \ref{thm: vcdim_translates} and \ref{thm: vcdim_graphs} is number theoretic. 

Let us use the terminology of Theorem \ref{thm: vcdim_graphs}; taking $G = \mathbb{Z}/N\mathbb{Z}$ and setting $p=1/2$ we find that for $\Gamma \sim \RC(\mathbb{Z}/N\mathbb{Z},1/2)$, asymptotically almost surely we have
\[ \VCdim(\Gamma) = (1+o(1))\log_2 N.\] 
(In fact for a sequence of random graphs $\Gamma_N \sim \RC(\mathbb{Z}/N\mathbb{Z},1/2)$ chosen independently, this relation holds almost surely by Borel-Cantelli.)

This provides ancillary evidence for the behavior of quadratic residues in $\mathbb{Z}/N\mathbb{Z}$ where $N$ is prime. If $S$ is deterministically chosen to be the set of quadratic residues in a prime field, the graph $\Gamma = \Cay(\mathbb{Z}/N\mathbb{Z},S)$ is called the Paley (di)graph. This is a well-known example of a deterministic family of graphs which behave pseudorandomly, having many of the same properties as the balanced Erd\H{o}s-R\'enyi graph \cite{chunggrahamwilson}. Conversely, it is reasonable to expect that for many problems, the behavior of Paley graphs may be conjecturally understood by solving the corresponding problem for the graph model $\RC(\Z/N\Z,1/2)$. To this end, Theorem~\ref{thm: vcdim_graphs} lends credence to the following conjecture, which is a special case of \cite[Conjecture~1.3]{MSW}.
\begin{conjecture}[McDonald--Sahay--Wyman]
\label{conj: quadres}
Let $\Gamma = \Gamma(N)$ be the Paley (di)graph on $\Z/N\Z$ with $N$ prime. Then, as $N \to \infty$ through the primes, 
\[ \VCdim(\Gamma) = (1 + o(1))\log_2 N. \]
\end{conjecture}
For more on this conjecture, including some partial progress and numerical evidence, see \cite{MSW}.

Furthermore, if $\Gamma \sim \RC(\mathbb{Z}/N\mathbb{Z}, 1/r)$, then asymptotically almost surely $\VCdim(\Gamma) = (1+o(1))\log_r N$. This would suggest that for $r \in \N$, $r \geqs 3$ such that\footnote{The congruence condition $r \mid N - 1$ is no imposition, since if $s \nmid N-1$, the set of $s$-th powers in $\Z/N\Z$ is the same as the set of $r$-th powers where $r = (s,N-1)$ is the gcd.} $N \equiv 1 \pmod r$, the Cayley graph generated by the set of $r$-th powers in $\mathbb{Z}/N\mathbb{Z}$ may have VC-dimension asymptotic to $(1+o(1))\log_r N$ as well. This is in contradistinction to \cite[Conjecture~1.5]{MSW}, where the authors conjecture that it must be asymptotic to $(1+o(1))\log_2 N$. Motivated by this discrepancy, we raise the following alternative conjecture.

\begin{conjecture}
\label{conj: higherres}
Let $r \geqs 3$ be a fixed integer and for $N \equiv 1 \pmod r$ prime, let $\Gamma = \Gamma(N)$ be the Cayley graph\footnote{with respect to the additive group of $\Z/N\Z$} $\Cay(\Z/N\Z,S)$ where $S \subseteq (\Z/N\Z)^\times$ is the multiplicative subgroup of $r$-th powers. Then, as $N \to \infty$ through the primes congruent to $1$ modulo $r$, 
\[ \VCdim(\Gamma) = (1 + o(1))\log_r N. \]
\end{conjecture}

We expect the natural generalization of the above to $\F_q$ for a prime power $q$ with $r \mid q-1$ to also hold as $q \to \infty$. In the notation of \cite[Equation~3]{MSW}, we conjecture that $\overline{\alpha}^{(r)} = \underline{\alpha}^{(r)} = \log_r 2$. In particular, the lower bound $\underline{\alpha}^{(r)} \geqs \tfrac12 \log_r 2$ implicit in \cite[Theorem~1.6]{MSW} cannot be improved by more than a factor of $2$ (see also, the discussion in \cite[Section~5.1]{MSW}). 

\subsection{A discussion of the proof}

Let us give a very brief and heuristic discussion of Theorems \ref{thm: vcdim_translates} and \ref{thm: vcdim_graphs} before treating them in more depth. By symmetry we may suppose $p\leqs 1/2$ so $r = 1/p$. In order to show that a random subset $A\subseteq G$ Bernoulli sampled with parameter $p$ has VC-dimension near $\log_r N$, we must show (i) that there is some subset of size slightly smaller than $\log_r N$ which is shattered by the family of translates of $A$, and (ii) that all subsets of size slightly larger than $\log_r N$ will not be shattered by this family. 

Let us fix some subset $U$ of size $k$ in $G$. By analogy with the classical coupon collector problem, let us say that we have `collected' a subset $K \subseteq U$ if $K$ is cut out from $U$ by $\mathcal{T}_A$, and let us furthermore make the (not quite correct) supposition that for $g_1, g_2,... \in G$, the subsets of $K$ collected by $g_1 A \cap U, g_2 A \cap U ,...$ will be collected independently at each step. Then at each step the subset of $U$ which will be least likely to have been collected will be $U$ itself, and moreover the probability of collecting $U$ will be $p^k$. Making this assumption of independence then, it should be the case that if $N = |G|$ is slightly larger than $1/p^k$, there should be a high likelihood of collecting $U$ and all other subsets of $U$ in these $N$ steps (so that $U$ is shattered). On the other hand, if $N$ is slightly smaller than $1/p^k$ there should be a high likelihood that $U$ is not collected and $U$ is thus not shattered. This is the claim made in the paragraph above. 

Of course we do not really have independence in the way utilized above. But by means of reasonably efficient estimates for tiling or covering the group $G$ by translates of a subset, we will be able to treat those dependencies that arise by means of union bounds.

We have already noted the analogy between this problem and the classical coupon collector problem, and there is a perhaps even closer similarity to problems of waiting times for runs in coin tossing. We comment further on these analogies in Section \ref{sec: concluding}.

\subsection*{Notation}
The subscript in $\log_r$ for $r > 1$ denotes the base of the logarithm. We assume the sampling probability $p$ always satisfies $0 < p < 1$ and write $r = 1/p > 1$. We use standard asymptotic notation $\ll, \gg, \asymp, O(\cdot), o(\cdot)$. The implicit constant or rate of decay in asymptotic notation may depend on $p$ (and hence $r$) but will otherwise be uniform in all parameters. We also adopt the notation that $c$ and $C$ refer to unspecified positive constants that may vary from line to line and may depend on $p$.

\subsection*{Acknowledgments}
We thank Jeff Lagarias, Brian McDonald, Ilya Shkredov, Olof Sisask, Caroline Terry, Julia Wolf, and Emmett Wyman for useful comments, references, or discussions. We also thank Dave Benson, Corentin Bodart, Sean Eberhard, Dave Rusin, and Yiftach Barnea for their comments on MathOverflow question \href{https://mathoverflow.net/q/483398}{\#483398} (accessed: 05/21/2025). Finally, we thank the anonymous referees for their comments. This paper grew out of conversations held in visits to Purdue University and Queen's University; we are grateful to both institutions for their hospitality.

\section{Tiling and covering lemmas}
\label{sec: tilingandcovering}

We will need the following two combinatorial lemmas.

\begin{lemma} 
\label{lem: tilingv2}
Let $G$ be a finite group of cardinality $N$ and suppose $U$ is a subset of $G$ of cardinality $k$. Then, $G$ contains a subset $S = \{s_1,...,s_\ell\}$ with cardinality $\ell$ satisfying $\ell \geqslant N/k^2$ such that the sets
\begin{equation} \label{Sprop} s_1U,\, s_2U,\, \cdots, s_\ell U \text{ are pairwise disjoint}. \end{equation}
\end{lemma}

\begin{proof} The proof is via a greedy construction. Let $S$ be a maximal subset of $G$ with the property \eqref{Sprop}. From this, it follows that for every $g \in G$, there exists an $s \in S$ such that $sU \cap gU \neq \emptyset$, as otherwise one may add $g$ to $S$ without violating \eqref{Sprop}. In other words, there exist $u_1,u_2 \in U$ such that $su_1 = gu_2$ exists in $G$. This implies that $g = su_1u_2^{-1} \in SUU^{-1}$, where $U^{-1}$ is the set of inverses of elements in $U$ and we write $AB$ for the product set of $A$ and $B$. Letting $V = UU^{-1}$, we see that we have shown that $G = SV$ and $S$ satisfies \eqref{Sprop} and it remains to show the estimates on the cardinality. We have the trivial inequality $|AB| \leqslant |A||B|$ for the cardinality of the product set. From this, it follows that \[ |V| \leqs |U||U^{-1}| = k^2 \] and further, 
\[ \ell = |S| \geqs |G|/|V| \geqs N/k^2, \]
as desired.
\end{proof}

In our applications $k$ will have a size roughly proportional to $\log N$ and $\ell$ will be as large as possible, so that $\ell$ and $k\ell$ will still nearly be of size $N$, up to logarithmic factors. For this reason the reader may think of the family $s_1U,\, s_2U,\, \cdots, s_\ell U$ as an approximate ``tiling" of $G$ in the applications that will follow.


We will also use the following lemma of Bollob\'{a}s, Janson, and Riordan for coverings of a group by translates.

\begin{lemma}
\label{lem: covering}
Let $G$ be a finite group of cardinality $N$ and suppose $S$ is a subset of $G$ of cardinality $\ell$. Then $G$ contains a subset $T = \{t_1,...,t_m\}$ with cardinality $m$ satisfying $m \leqs \frac{N}{\ell}(\log \ell + 1)$ such that
\[
G = ST = \bigcup_{i=1}^m S t_i.
\]
\end{lemma}

\begin{proof}
This is Corollary 3.2 in \cite{bollobasjansonriordan} (with the convention there of left-multiplication by elements $t_i$).
\end{proof}

Note that the proof in \cite{bollobasjansonriordan} is also via a greedy construction. 

In the application below, we will use Lemma~\ref{lem: tilingv2} to construct a set and then apply Lemma~\ref{lem: covering} when $S$ is the set of its inverses, but the reader may note that Lemma~\ref{lem: covering} holds for arbitrary $S$. 

\begin{remark}
Lemma~\ref{lem: tilingv2} can be seen as dual to a special case of the Ruzsa covering lemma (see, for example, \cite[Lemma 3.6]{tao_nonabrusza} with $A = U$, $B = G$, and $K = N/k$ in the ``resp.'' version), which states that $G$ can be covered by $\ell \leqs N/k$ left-translates of $V$. Indeed, both proofs proceed via the same greedily constructed set $S$. We thank an anonymous referee for pointing this out to us. \end{remark}

\begin{remark}
If $G$ is abelian, then the construction of $S$ in Lemma~\ref{lem: tilingv2} leads easily to Lemma~\ref{lem: covering} for the set of inverses. For in this case $G = G^{-1} = S^{-1}V^{-1}$, whence we can simply set $T = V^{-1}$. We thank Ilya Shkredov for this remark, and for a suggestion that helped us simplify the exposition for Lemma~\ref{lem: tilingv2}.
\end{remark}
\section{The lower bound for \texorpdfstring{$\VCdim(A)$}{VCdim A} in Theorem \ref{thm: vcdim_translates}}
\label{sec: lower}

The reader may check that Theorem \ref{thm: vcdim_translates} is symmetric in $p$ and $1-p$, so in a proof we may assume $p \leqs 1/2$, and thus $r = 1/p$. We carry out the proof of the upper and lower bounds for VC-dimension separately.

Thus in this section we prove the lower bound: that
\begin{equation}
\label{eq: lowerbound}
\VCdim(A) \geqs \log_r N - 10 \log_r \log_r N
\end{equation}
with probability $1 - O(1/N^\eta)$ for any fixed $\eta>0$.

Take $k = \lceil \log_r N - 10 \log_r \log_r N\rceil$ so that $N \asymp k^{10} r^k$. We take an arbitrary subset $U \subset G$ of cardinality $k$. We will prove the estimate \eqref{eq: lowerbound} by showing the likelihood $U$ is \emph{not} shattered is very small. We divide the proof into two steps.

\vspace{5pt}

\noindent \textbf{Step 1:} We observe that uniformly for all subsets $\emptyset \subseteq K \subseteq U$, 
\[
\Pr(K\text{ is not cut out of $U$ by }\T_A) \ll e^{-Ck^8},
\]
for an implicit constant and a constant $C$ which depends only on $p$.

We demonstrate this claim as follows. Let $\ell \geqs N/k^2$ and $S = \{s_1, \ldots, s_\ell\}$ be the collection of elements guaranteed by Lemma~\ref{lem: tilingv2} such that $s_1 U,\ldots, s_\ell U$ are pairwise disjoint. For any random set $A$, define 
\[
\T_A' = \{s_i^{-1} A:\, 1\leqs i \leqs \ell\}
\] 
to be a subfamily of $\T_A$.

Then we have,
\begin{align}
\label{eq: cutoutprod}
\notag \Pr[K\text{ is not cut out of }U\text{ by }\T_A] &\leqs \Pr[K\text{ is not cut out of }U\text{ by }\T'_A] \\
\notag &= \Pr\left[\bigwedge_{1 \leqs i \leqs \ell} K \neq s_i^{-1}A \cap U \right]\\
\notag &= \Pr\left[\bigwedge_{1 \leqs i \leqs \ell} s_i K \neq A \cap s_i U \right] \\
&= \prod_{1\leqs i \leqs \ell} \Pr\left[ s_i K \neq A \cap s_i U\right],
\end{align}
with the last step following because each event in the product will be independent, as the sets $s_i U$ are pairwise disjoint.

Note that
\[
\Pr\left[ s_i K = A \cap s_i U \right] = p^{|K|}(1-p)^{k-|K|} \geqs p^k,
\]
because $p \leqs 1-p$ and $|U| = k$. Recalling that $r = 1/p$, we see that the right hand side of \eqref{eq: cutoutprod} is
\[
\leqs (1-r^{-k})^{\ell}.
\]
Since $\ell \gg k^8 r^k$, a simple calculus exercise yields that the right hand side of the above is $O\big(e^{-Ck^8}\big)$ as claimed.

\vspace{5pt}

\noindent \textbf{Step 2:} We leverage the estimate in Step 1 through a crude union bound. We use the same subset $U$ as in Step 1. Note
\begin{align*}
\Pr[\VCdim(A) < k] &\leqs \Pr[U \text{ is not shattered by }\T_A]\\
&= \Pr\left[\bigvee_{\emptyset \subseteq K \subseteq U} K \text{ is not cut out of } U \text{ by } \T_A\right]\\
&\leqs \sum_{\emptyset \subseteq K \subseteq U} \Pr\left[K \text{ is not cut out of } U \text{ by }\T_A\right]\\
&\ll 2^k e^{-Ck^8}.
\end{align*}

Since $k \asymp \log N$ the above is $O(1/N^\eta)$ for any fixed $\eta > 0$. Recalling the definition of $k$, we see that we have proved \eqref{eq: lowerbound}.

\section{The upper bound for \texorpdfstring{$\VCdim(A)$}{VCdim A} in Theorem \ref{thm: vcdim_translates}}
\label{sec: upper}

We turn to the upper bound, showing now that
\begin{equation}
\label{eq: upper bound}
\VCdim(A) \leqs \log_r N + 10 \log_r \log_r N
\end{equation}
with probability $1 - O(1/N^\eta)$ for any fixed $\eta>0$.

As above we lose no generality in assuming $p\leqs 1/2$, so that $r=1/p$.

For the upper bound case, we take $k = \lfloor \log_r N + 10 \log_r \log_r N \rfloor$, so that $N \asymp r^k/k^{10}$. Again we divide our proof into two steps, beginning by showing for these parameters if $U \subseteq G$ is an arbitrary subset of cardinality $k$ the likelihood that $U$ \emph{is} shattered is very small.

\vspace{5pt}

\noindent \textbf{Step 3:} We show there is a constant $C>0$ such that for any $U \subseteq G$ of cardinality $k$,
\begin{equation}
\label{eq: shatteredissmall}
\Pr[U\text{ is shattered by }\T_A] \ll e^{-Ck^6}
\end{equation}

The outline of this step is as follows: using lemmas for covering and tiling, we show that the family $\T_A$ can be replaced by the union of a small number of families of the sort $\{\sigma_1 A, \ldots, \sigma_\ell A\}$ where $\sigma_1^{-1} U, \ldots, \sigma_\ell^{-1} U$ are pairwise disjoint. If $U$ is shattered by $\T_A$, we show some such family must cut out a large number of subsets of size nearly $k$. But exploiting independence, we show the probability of this is small.

In more detail, we set up our tiling and covering as follows. Using Lemma \ref{lem: tilingv2} we can choose $s_1,\ldots, s_\ell$ such that $s_1 U, \ldots , s_\ell U$ are pairwise disjoint with $\ell = \lceil N/k^2 \rceil$. Now set $S^{-1} = \{s_1^{-1},\ldots,s_\ell^{-1}\}$. 

We now use Lemma \ref{lem: covering} to choose $t_1, \ldots, t_m$ with 
\[
m\leqs \tfrac{N}{\ell}(\log \ell+1) \ll k^3
\] 
such that $G = \bigcup_{j=1}^m S^{-1}t_j$. (The implicit constant above may depend on $r$, but we consider $r$ fixed.) We define the families $\T_A^{(1)},\ldots, T_A^{(m)}$ by 
\[
\T_A^{(j)} = \{gA:\, g \in S^{-1} t_j\}
\]
and note that this implies
\begin{equation}
\label{eq: Tj_covers}
\T_A = \bigcup_{j=1}^m \T_A^{(j)},
\end{equation}
and for $1 \leqs j \leqs m$,
\begin{equation}
\label{eq:indep_trials}
t_j^{-1}s_1 U\,, \ldots, \,t_j^{-1} s_\ell U\; \text{are disjoint.}
\end{equation}

For reasons that will be apparent momentarily, let us restrict our attention to subsets of cardinality $k-10$ which are cut out of $U$. We have
\begin{equation}
\label{eq: cardinality k-10}
\Pr[U \text{ is shattered by } \T_A] \leqs \Pr\left[ \bigwedge_{\substack{K \subseteq A \\ |K| = k-10}} K \text{ is cut out of } U \text{ by } \T_A\right]. 
\end{equation}

There are $\binom{k}{10}$ subsets $K \subseteq U$ of cardinality $k-10$. On the right hand side of \eqref{eq: cardinality k-10}, we are considering an event in which all $\binom{k}{10}$ are cut out by $\T_A$, so by the pigeonhole principle, for some $j$ we must have that $\T^{(j)}_A$ cuts out at least $\binom{k}{10}/m$ subsets of $U$ of cardinality $k-10$. For sufficiently large $k$ we have $\binom{k}{10}/m \geqs k^6$. Therefore for sufficiently large $k$, the right hand side of \eqref{eq: cardinality k-10} is bounded by
\begin{equation}
\label{eq: cutsout k^6}
\Pr\left[\bigvee_{j=1}^m \; \T_A^{(j)}\text{ cuts out at least } k^6 \text{ subsets of } U \text{ of cardinality } k-10 \right],
\end{equation}

Fixing $j$, we abbreviate $t = t_j$, and consider the event:
\begin{equation}
\label{eq: event_description}
\T_A^{(j)}\text{ cuts out at least } k^6 \text{ distinct subsets of } U \text{ of cardinality } k-10.
\end{equation}
If this event occurs and these subsets are labeled $H_1,\ldots,H_h$ for $h \geqs k^6$, we must have for each $1\leqs u \leqs h$
\[
H_u = (s_v^{-1} t A) \cap U\quad \text{for some } v = v(u).
\]
But this is the same as
\[
t^{-1} s_v H_u = A \cap (t^{-1} s_v U)\quad \text{for some } v = v(u).
\]
Furthermore, for distinct indices $u$, the corresponding indices $v$ must be distinct, for otherwise the sets $H_u$ would coincide.

Thus, when the event \eqref{eq: event_description} occurs, we must have
\[
|A \cap (t^{-1} s_v U)| = k-10
\]
for at least $k^6$ values of $v$. Define the indicator random variables
\[
X_v = \mathbf{1}(|A \cap (t^{-1} s_v U)| = k-10), \quad 1 \leqs v \leqs \ell.
\]
The probability of the event \eqref{eq: event_description} is thus no more than
\[
\Pr[X_1+\cdots + X_\ell \geqs k^6]
\]

By \eqref{eq:indep_trials}, the random variables $X_1,\ldots,X_\ell$ are independent Bernoulli random variables, and we have $\E\, X_v = \binom{k}{10} p^{k-10} (1-p^{10}) \leqs k^{10} p^{k-10}$, so that 
\[
\E[X_1+\cdots + X_\ell] \leqs \ell k^{10} p^k \leqs ck^{-2},
\]
for some constant $c > 0$. Chernoff's inequality\footnote{Chernoff's inequality tells us that for Bernoulli random variables $Y_i$ with parameters $p_i$, if $\E[Y_1+\cdots + Y_m] = \mu,$ then for $t \geqs \mu$, we have 
\[
\Pr[Y_1+\cdots + Y_m \geqs t] \leqs e^{-\mu} (e\mu/t)^t.
\]
See \cite[Theorem 2.3.1]{vershynin}.} therefore yields
\[
\Pr[X_1+\cdots + X_\ell \geqs k^6] \leqs (e c k^{-2}/k^6)^{k^6} \ll \exp(-C k^6 \log k),
\]
for a constant $C > 0$. Since this holds for each $j$, by a union bound \eqref{eq: cutsout k^6} is bounded by
\[
\leqs m \cdot \exp(- C k^6 \log k) \ll e^{-C k^6}.
\]
This proves the claim \eqref{eq: shatteredissmall}.

\vspace{5pt}

\noindent \textbf{Step 4:} We use the shattering bound proved in Step 3 to deduce that it is very unlikely the VC-dimension of $A$ is greater than or equal to $k$. For note that
\begin{align*}
\Pr[\VCdim(A) \geqs k] &= \Pr\left[\bigvee_{\substack{U \subseteq G \\ |U| = k}} U \text{ is shattered by } \mathcal{T}_A\right] \\
&\leqs \sum_{\substack{U \subseteq G \\ |U| = k}} \Pr[U \text{ is shattered by } \mathcal{T}_A] \\&
\ll \binom{N}{k} e^{-Ck^6} \ll r^{k^2} e^{-Ck^6}.
\end{align*}
Because $k \asymp \log N$, this is $O(1/N^\eta)$ for any $\eta>0$, which implies the upper bound \eqref{eq: upper bound} we desired.

\section{Concluding remarks and further questions}
\label{sec: concluding}

In this paper we give, for arbitrary finite groups, a law of large numbers for the VC-dimension of the set family arising from translations of a random subset (equivalently for the set family of neighborhoods of a random Cayley graph induced by this subset). One natural question raised by this work remains the behavior of the VC-dimension of (translates of) pseudorandom deterministic subsets, such as second-and-higher-power residues as discussed in \S\ref{subsec: numtheory}.

We have also worked exclusively with Bernoulli sampled random subsets. There is another natural model for a random subset $A$, namely for a parameter $d$, samples all subsets of size $d$ uniformly, as discussed in Remark~\ref{rem: graphmodels}. One may expect this leads to the same answer we have obtained here where $d/N = p$, but we do not pursue this here.

Finally, we have already mentioned the resemblance between this problem and the classical coupon collector problem -- which asks one to characterize the number of draws necessary to collect all of a collection of $m$ types of coupons, if the probability of collecting each type in each draw is uniform or arising from some other distribution -- as well as the problem of waiting times for runs in coin tossing -- which asks one to characterize the length of time necessary to wait to see all binary patterns of length $k$ in repeated (possibly biased) coin tosses. In fact in both cases not only are laws of large numbers known, but limiting distributions have been found as well (see e.g. \cite{erdosrenyi,neal} for the coupon collectors problem, and \cite{mori1991,mori1995} for runs of coins). It would be interesting if limit laws can be found for $\VCdim(A)$ also. It may indeed be that for a wide variety of groups, the quantity $\VCdim(A)$ concentrates even more sharply around a value than Theorems \ref{thm: vcdim_translates} and \ref{thm: vcdim_graphs} indicate.

\bibliography{VCdim}
\bibliographystyle{plain}
\vspace{0.5cm}
\end{document}